\documentclass[preprint,12pt]{elsarticle}

\usepackage{amssymb}
\usepackage{amsmath}
\usepackage{amsfonts}
\usepackage{graphicx}
\usepackage{graphics}
\usepackage{tikz}
\usepackage{titlesec}

\newtheorem{theorem}{Theorem}

\newtheorem{corollary}[theorem]{Corollary}

\newtheorem{remark}[theorem]{Remark}

\def\qed{\vbox{\hrule
 \hbox{\vrule\hbox to 5pt{\vbox to 8pt{\vfil}\hfil}\vrule}\hrule}}

\usepackage{tikz}
\usepackage{tkz-graph}

\journal{xxxxxxxxxxxx}

\begin{document}

\begin{frontmatter}

\title{New lower bounds for the energy of matrices and graphs}

\author{Enide Andrade}
\ead{enide@ua.pt}
\address{CIDMA-Center for Research and Development in Mathematics and Applications
         Department of Mathematics, University of Aveiro, 3810-193, Aveiro, Portugal.}
\author{Juan R. Carmona}
\address{Facultad de Ciencias -- Instituto de Ciencias F\'{i}sicas y Matem\'{a}ticas\\
 Universidad Austral de Chile, Independencia 631 \\ Valdivia - Chile. }
\ead{juan.carmona@uach.cl}

\author{Geraldine Infante}
\ead{geraldine.infante@alumnos.ucn.cl}
\address{Departamento de Matem\'{a}ticas, Universidad Cat\'{o}lica del Norte
         Av. Angamos 0610 Antofagasta, Chile.}
\author[]{Mar\'{\i}a Robbiano \corref{cor1}}
\ead{ mrobbiano@ucn.cl}
\address{Departamento de Matem\'{a}ticas, Universidad Cat\'{o}lica del Norte
         Av. Angamos 0610 Antofagasta, Chile.}
\cortext[cor1]{Corresponding author}
\begin{abstract}

\noindent Let $R$ be a Hermitian matrix. The energy
of $R$, $\mathcal{E}(R)$, corresponds to the sum of the absolute values of its eigenvalues. In this work it is obtained two lower bounds for $\mathcal{E}(R).$ The first one generalizes a lower bound obtained by Mc Clellands for the energy of graphs in $1971$ to the case of Hermitian matrices and graphs with a given nullity. The second one generalizes a lower bound obtained by K. Das, S. A. Mojallal and I. Gutman in 2013 to symmetric non-negative matrices and graphs with a given nullity. The equality cases are discussed. These lower bounds are obtained for  graphs with $m$ edges and some examples are provided showing that, some obtained bounds are incomparable with the known lower bound for the energy $2\sqrt{m}$.
Another family of lower bounds are obtained from an increasing sequence of lower bounds for the spectral radius of a graph. The bounds are stated for singular and non-singular graphs.

\end{abstract}


\end{frontmatter}

\section{Notation and Preliminaries}
\noindent
In this work we deal with an $(n,m)$-graph $G$ which is an undirected simple graph with vertex set $V\left( G\right) $ of cardinality $n$ and edge set $E\left( G\right)$ of cardinality $m$.
As usual we denote the adjacency matrix of $G$ by $A=A(G)$. The eigenvalues of $G$ are the eigenvalues of $A$ (see e.g.\cite{C-D-S1,C-D-S2}). Its eigenvalues will be denoted (and ordered) by  $ \lambda_{1} \geq \cdots \geq \lambda_{n}.$ We denote the spectrum of a graph $G$ by $\sigma(G)= \sigma (A(G)).$
\noindent If $e\in E(G)$ has end vertices $i$ and $j$ then it is denoted by $ij$. If $i\in V(G),$ $N_{G}(i)$ denotes  the set of neighbors of the vertex $i$ in $G.$  For the $i$-th vertex of $G$, the cardinality of $N_{G}(i)$ is the degree of $i$ and it is denoted by either $d(i)$ or $d_{i}$. The number of walks of length $k$ of $G$ starting at $i$ is referred as the $k$-degree of the vertex $i$ and is denoted by $d_{k}(i)$ (see \cite{D-G}). For convenience, we set
\begin{align*}
  d_0(i) &= 1, \  d_1(i) = d(i), \ \text{and} \\
 d_{k+1}(i)&= \sum_{j \in N_{G}(i)}d_k(j).
\end{align*}
If $G$ is a connected graph,  then $A(G)$  is  a non-negative  irreducible  matrix \cite{C-D-S1}. The complement of a graph $G$ is usually denoted by $\overline{G}.$
A graph $G$ with $n$ vertices is called a \textit{regular} graph (or $r$-\textit{regular}) if $d_{i} = r, \ 1\leq i \leq n.$
A star and the complete graph with $n$ vertices is denoted by $S_{n}$ and $K_{n},$ respectively.
We recall now some concepts from Matrix Theory used throughout the text.
In this paper $R$ stands for a Hermitian complex matrix of order $n$ and $M$ represents any square complex matrix. It is well known that for a Hermitian matrix its singular values and the absolute values of its eigenvalues coincide. The energy of $R$, denoted by $\mathcal{E}\left(R\right),$ is the sum of the absolute values of the eigenvalues of $R$.
Note that,
if $R$ is a non-negative matrix, then $R$ is symmetric and its spectral radius, $\rho= \rho(R)$, and its largest eigenvalue coincide.

\noindent For an arbitrary square matrix $M$ of order $k$ with eigenvalues $\mu_1,\ldots, \mu_k$, its nullity, denoted by $\eta(M)$, corresponds to the multiplicity of its null eigenvalue. Thus, if $M$ is non-singular then $\eta(M)=0$.
Note that, for a graph $G$, the nullity of $A(G)$ is called the nullity of $G$ and it is denoted by $\eta(G).$ Consequently, a graph $G$ is called non-singular if $\eta(G)=0$ otherwise, $G$ is called singular. In the text we denote by $\mathbf{e}$ the all ones vector.

\noindent The $k$-th elementary symmetric sum of the eigenvalues $\mu_1, \mu_2, \ldots, \mu_n$ of a square matrix $M$ of order $n$ is defined as
 \begin{align}
\Upsilon_k\left(M\right) =\sum\limits_{1\leq i_{1}<i_{2}<\cdots <i_{k }\leq n}\mu_{i_{1}}\mu_{i_{2}}\cdots\mu_{i_{k}}.
\end{align}
Note that $\Upsilon_n\left(M\right)=\det(M) $ and $\Upsilon_1\left(M\right)=\rm{tr}(M),$ with $\rm{tr} (.)$ denoting the trace of a square matrix.
For a square matrix $M$ of order $n,$ let $M[i_{1},i_{2},\ldots, i_{k}]$ be the principal submatrix of $M$ whose $j$-th row and column are labeled by $i_{j}$, for $1\leq j \leq n$. Then, $\det \left( M[i_{1},i_{2},\ldots, i_{k}]\right)$ is a principal minor of order $k$ of $M$ and it is denoted by $\Delta_{M} \left(i_{1},i_{2},\ldots, i_{k}\right)$. In \cite{CDM} it is shown that
\begin{equation}
\label{expansions}
\left|\Upsilon_k\left(M\right)\right| =\left|\sum\limits_{1\leq i_{1}<i_{2}<\cdots <i_{k}\leq n} \Delta_{M} \left(i_{1},i_{2},\ldots, i_{k}\right)\right|.
\end{equation}

\noindent The Frobenius matrix norm of a square complex matrix $M$, denoted by $\left|M\right|,$ is defined as the square root of the sum of the squares of its singular values. In consequence, if $R$ is a symmetric matrix of order $n$ with eigenvalues $\alpha_{1},\alpha_{2},\ldots ,\alpha_{n},$ $$\left|R\right|^{2}=\sum\limits_{i=1}^{n}\left|\alpha_i\right|^{2}.$$
The paper is organized as follows. At Section 2 some motivation in connection with Chemistry and known lower bounds for $\mathcal{E}(G)$ and the main results without proof are introduced. At Section 3 three cases where the lower bound $2 \sqrt{m}$ introduced by Caporossi et al. in \cite{C-C-G-H} is improved by the lower bound at Theorem \ref{CR}, are presented.
At Section 4 the main theorems and corollaries presented at Section 2 are proved. Namely, in this section one lower bound for $\mathcal{E}\left(R\right)$ is given and generalizes the lower bound for the energy in \cite{MC} to the case of Hermitian matrices with a given nullity.
In \cite{C-G-T-R} an increasing non-negative sequence that converges to the spectral radius of a non-negative symmetric matrix was constructed and a decreasing sequence of upper bounds for the energy of $R$ was obtained. Therefore, using the same sequence, an increasing sequence of lower bounds for $\mathcal{E}\left(R\right),$ where $R$ has given nullity, is obtained at Section 5. Moreover, some results are applied to the adjacency matrix of a graph to obtain lower bounds for the energy of graphs. Equality cases are studied.


\section{Motivation and the main results}
\noindent The concept of energy of graphs appeared in Mathematical Chemistry and we review in this section its importance. In Chemistry the structure of molecules are represented by \emph{molecular graphs} where its vertices stand for atoms and edges for bonds. Molecular graphs can be split into two basic types: one type representing \emph{saturated hydrocarbons} and another type representing \emph{conjugated }$\pi$ \emph{-electron systems}. In the second class, the molecular graph should have perfect matchings (called ``Kekul\'{e} structure"). In the 1930s, Erich H\"{u}ckel put forward a method for finding approximate solutions of the Schr\"{o}dinger equation of a class of organic molecules, the so-called \emph{conjugated hydrocarbons} (conjugated $\pi$-electron systems) which have a system of connected $\pi$-orbitals with \emph{delocalized} $\pi$-electrons (electrons in a molecule that are not associated with a single atom or a covalent bond). Thus, the HMO (H\"{u}ckel molecular orbital model) enables to describe approximately\ the behavior of the so-called $\pi$-electrons in a conjugated molecule, especially in conjugated hydrocarbons.
For more details see \cite{gE} and the references therein.
Following to HMO theory, the total $\pi$-electron energy, $E_{\pi },$ for conjugated hydrocarbons in their ground electronic states, $E_{\pi }$ is calculated from the eigenvalues of the adjacency matrix of the molecular graph:
\begin{equation*}
\mathcal{E}_{\pi }=n\alpha +\mathcal{E}\beta ,
\end{equation*}%
where $n$ is the number of carbon atoms, $\alpha $ and $\beta $ are the HMO
carbon-atom coulomb and carbon-carbon resonance integrals, respectively.
For
the majority conjugated $\pi $-electron systems%
\begin{equation}
\mathcal{E}=\sum_{i=1}^{n}\left\vert \lambda _{i}\right\vert ,  \label{tradener}
\end{equation}%
where $\lambda _{1},\ldots ,\lambda _{n}$ are the eigenvalues of the
underlying molecular graph. For molecular structure researches, $\mathcal{E}$ is a
very interesting quantity. In fact, it is traditional to consider $\mathcal{E}$ as the
total $\pi $-electron energy expressed in $\beta $-units. The spectral
invariant defined by (\ref{tradener}) is called the \emph{energy} of the
graph $G$, and it will be denoted here by $\mathcal{E}(G)$ (see \cite{IG1}). It is worth to be mentioned that in the contemporary literature this graph invariant is widely studied, namely the search for its upper bounds. On the other hand, lower bounds for energy are much fewer in number, probably because these are much more difficult to deduce. Some of these, recently determined, the reader should be referred, for instance, to \cite{ Agudelo,Bozkurt, Ji, Marin, Tian}.

\noindent For an arbitrary graph $G,$ in \cite{MC} McClellands obtained the following lower bound for $\mathcal{E}(G)$:
\begin{align}
  \label{mcc}
\mathcal{E}(G) \geq\sqrt{2m + n(n -1) \left|\det(A) \right|^{2/n}}. \end{align}
where $\det(A)$ denotes the determinant of the matrix $A=A(G).$
The following simple lower bound for a graph $G$ with $m$ edges was introduced by Caporossi et al. in \cite{C-C-G-H} and the equality case was discussed. In fact,
\begin{align}
\label{superior}
  \mathcal{E}(G) \geq  2 \sqrt{m},
\end{align}
\noindent with equality if and only if $G$ consists of a complete bipartite graph $K_{a,b}$ such that $ab= m$ and arbitrarily many isolated vertices. A lower bound for the energy of symmetric matrices and graphs was introduced in \cite{A-R-S}. Necessary conditions for the equality were studied. Some computational experiments were presented shown that, in some cases, the obtained lower bound is incomparable with the lower bound $2\sqrt{m}$.

\noindent In \cite{D-M-G}, Das et al. obtained the following lower bound for a connected non-singular $(n,m)$-graph:
\begin{align}
\label{DasMojallalGutman}
\mathcal{E}(G )\geq \frac{2m}{n} + (n-1) + \ln | \det A |- \ln \frac{2 m}{ n} ,
\end{align}
where $\det (A) $ denotes the determinant of the adjacency matrix $A= A(G).$
\noindent The equality holds in (\ref{DasMojallalGutman}) if and only if $G$ is  the complete graph $K_n$.
The last lower bound was obtained firstly considering that, for a connected graph, the following relationship holds:
\begin{align}\label{2}
   \mathcal{E}(G) \geq \lambda_1 +(n-1) +\ln|\det A|-\ln{\lambda_1}.
\end{align}
In \cite{D-M-G} it was shown that the graph that attains equality in (\ref{2}) is the same graph that attains equality  in (\ref{DasMojallalGutman}).

\noindent  The present work generalizes the lower bound in (\ref{mcc}) for Hermitian matrices $R$, such that $\eta(R)=\kappa$ and the lower bound in (\ref{2}) for non-negative symmetric matrices $R$, such that $\eta(R)=\kappa$. The equality cases are discussed.

 \noindent We present now the main results of this work to be proven at Section 4. Additionally, using the increasing sequence of lower bounds for $\lambda_1$ given in \cite{C-G-T-R} an increasing sequence of lower bounds for the energy of graphs with nullity $\kappa$, is obtained at Section 5.

\begin{theorem}
\label{Robbiano-Infante-Carmona}
Let $R$ be a non-negative symmetric matrix such that $\eta(R)=\kappa$. Then
\begin{align}
\label{mcc2}
\mathcal{E}(R) \geq\sqrt{\left|R\right|^{2} + (n-\kappa)(n-\kappa -1) \left|\Upsilon_{n-\kappa}(R)\right|^{\frac{2}{n-\kappa}}},
\end{align}
The equality holds in (\ref{mcc2}) if and only if the nonzero eigenvalues of $R$ have the same absolute value. Moreover, if $R$ is irreducible the equality holds if and only if $R$ is permutationally equivalent to a block matrix of the form,
\begin{equation}
\label{bm}
  \left(\begin {matrix}
  0 & S\\
  S^{T} & 0
  \end{matrix}\right)
  \end{equation}
where $\kappa=n-2$ and $S$ is a rank one matrix.
\end{theorem}

\noindent Note that if in the above result the symmetric matrix $R$ is replaced by the adjacency matrix of a graph $G$ the following result is obtained.
\begin{theorem}
\label{CR}
Let $G$ be an $(n,m)$-graph without isolated vertices, with $\eta(G)=\kappa$. Then
\begin{align}
\label{mcc3}
\mathcal{E}(G) \geq\sqrt{2m + (n-\kappa)(n-\kappa -1) \left|\Upsilon_{n-\kappa}(G)\right|^{\frac{2}{n-\kappa}}},
\end{align}
where
\begin{equation*}
\Upsilon_{n-\kappa}(G) =\Upsilon_{n-\kappa}\left(A(G)\right).
\end{equation*}

\noindent The equality holds in (\ref{mcc3}) if and only if the nonzero eigenvalues of $G$ have the same absolute value. Moreover, if $G$ is connected the equality holds if and only if $G=K_{a,b}$ the complete bipartite graph, with $a+b=\kappa+2$. Otherwise $G=\cup_{j=1}^{\ell} K_{a_j,b_j},$ with $a_jb_j=a_ib_i, \ \text{for} \ i\neq j,$ $\ell=\frac{n-\kappa}{2}$ and $n=\sum_{j=1}^{\ell}(a_j+b_j).$
\end{theorem}



\begin{theorem}
\label{RCI2}
Let $R$ be a non-negative symmetric matrix of order $n$ with spectral radius $\rho$ such that $\eta(R)=\kappa.$ Then
\begin{align}
\label{DasMGutmanCarmonaRobbiano4}
\mathcal{E}(R)\geq \rho+n-\kappa-1+\ln{\left|\Upsilon_{n-\kappa}(R)\right|}-\ln{\rho}.
\end{align}
The equality holds in (\ref{DasMGutmanCarmonaRobbiano4}) if and only if the nonzero eigenvalues of $R$ have all modulus equal to $1$, except maybe for its largest eigenvalue. Moreover, if $R$ has largest eigenvalue greater than $1$ and $\rm{tr}(R)=0$ then $\kappa$,  the number $c$ of eigenvalues equal to $-1$ and, the number $f$ of eigenvalues equal to $1$ satisfy:

\begin{enumerate}
    \item \label{ce} $c=\frac{\left\vert  R\right \vert^{2}- \rho^{2}+ \rho}{2};$ \\

    \item \label{efe} $f=\frac{\left\vert  R\right \vert^{2}-\rho^{2}-\rho}{2};$\\

    \item \label{kappa} $\kappa=n-1+\rho^{2}-\left\vert  R\right \vert^{2}.$
\end{enumerate}
Moreover, the inequality (\ref{DasMGutmanCarmonaRobbiano4}) is strict if $R$ has a submatrix of order $3$, say $R_1$, where either
\begin{enumerate}

\item \label{subm1} $R_1=\left( \begin{matrix} 0 & a & 0\\a & 0 & b\\
    0 & b & 0\end{matrix}\right ) $ with $ \sqrt{a^{2}+b^{2}}>1,$ or

\item \label{subm2} $R_1=\left( \begin{matrix} 0 & a & c\\a & 0 & b\\
    c & b & 0 \end{matrix} \right )$ with a vector $(\alpha, \beta, \gamma)^{T}$ such that $$\frac{2(a\alpha\beta + b\beta\gamma+ c\alpha \gamma)}{\sqrt{\alpha^{2}+ \beta^{2}+ \gamma^{2}}}< -1.$$
\end{enumerate}


\end{theorem}

\noindent The result in (\ref{2}), can be generalized for all graphs, including singular graphs. The result is stated in Theorem \ref{rCI}.

\begin{theorem}
\label{rCI}
Let $G$ be a graph with $n$ vertices with  largest eigenvalue $\lambda_{1}$ and $\eta(G)=\kappa$. Then
\begin{align}\label{C-R0}
   \mathcal{E}(G) \geq \lambda _{1}+ n-\kappa-1   +\ln{\left|\Upsilon_{n-\kappa}(G)\right|}-\ln{\lambda_1}
\end{align}
The equality holds in (\ref{C-R0}) if and only if the nonzero eigenvalues of $G$, except maybe for its largest eigenvalue, have all modulus equal to $1$. If the largest eigenvalue of $G$ is $1$ then $G=\lfloor \frac{n-\kappa}{2}\rfloor K_2 \cup \kappa K_1.$ On the contrary, if $\rho >1 $ then $G=K_{n-\ell}\cup \kappa K_{1}\cup \lfloor\frac{\ell-\kappa}{2}\rfloor K_2$ with
$\kappa\leq \ell \leq n-3.$
\end{theorem}
\noindent As a consequence of Theorem \ref{RCI2} the  following result can be obtained.

\begin{corollary}
\label{RCI}
Let $R$ be a non-negative symmetric matrix of order $n$ with largest eigenvalue $\rho$ such that $\eta(R)=\kappa$ and there exists a non-negative vector $\mathbf{x}$ such that $$\mu=\frac{\mathbf{x}^{T} R \mathbf{x}}{\mathbf{x}^{T}\mathbf{x}} \geq 1.$$ Then
\begin{align}
\label{DasMGutmanCarmonaRobbiano5}
\mathcal{E}(R)\geq \mu+n-\kappa-1+\ln{\left|\Upsilon_{n-\kappa}(R)\right|}-\ln{\mu}.
\end{align}
The equality holds in (\ref{DasMGutmanCarmonaRobbiano5}) if and only if $\mathbf{x}$ is an eigenvector of $R$ associated to $\rho$ and all the nonzero eigenvalues of $R$ have absolute values equal to $1$, except maybe for its largest eigenvalue.
\end{corollary}

\begin{remark}
\label{regular}
If $R$ (with $R$ reducible) is partitioned into
irreducible blocks with one principal {main} block, say $W,$ whose spectral radius is  the spectral radius of $R,$ say $\rho$ such that $W\mathbf{y}=\rho\mathbf{y},$ then $R$ has an associated eigenvector $\mathbf{x}=(\mathbf{y},\mathbf{0},\ldots, \mathbf{0})^{T},$ and if all the nonzero eigenvalues of $R$ have absolute values equal to $1$, except maybe for its largest eigenvalue, the equality in (\ref{DasMGutmanCarmonaRobbiano5}) is also obtained.
\end{remark}


\noindent Therefore, for graphs, the result can be rewritten as follows:
 \begin{corollary}
\label{Carmona}
Let $G$ be an $(n,m)$-graph with largest eigenvalue $\lambda_{1}$ and let an induced $(n_1,m_1)$-subgraph, say $G_{1},$ where $G_1$ is any connected component with $n_1 \geq 2$. Therefore
\begin{align}
\label{DasMGutmanCarmonaRobbiano3}
\mathcal{E}(G) \geq \frac{2m_1}{n_1}+n-\kappa-1+\ln{\left|\Upsilon_{n-\kappa}(G)\right|}-\ln{\frac{2m_1}{n_1}}.
\end{align}
In particular, if $G_1$ is $r_1$-regular then
\begin{align}
\mathcal{E}(G) \geq r_1+n-\kappa-1+\ln{\left|\Upsilon_{n-\kappa}(G)\right|}-\ln{r_1}.
\end{align}
If $\rho=1$ then equality holds if and only if  $G=\lfloor \frac{n-\kappa}{2}\rfloor K_2 \cup \kappa K_1.$ On the contrary, if $\rho >1 $ then equality holds if and only if $G=K_{n-\ell}\cup \kappa K_{1}\cup \lfloor\frac{\ell-\kappa}{2}\rfloor K_2$ with
$\kappa\leq \ell \leq n-3,$ taken $G_1=K_{n-\ell}.$

\end{corollary}
\section{Three cases where the lower bound $2\sqrt{m}$ is improved by the lower bound at Theorem \ref{CR}}
\noindent In this section we present some cases where the lower bound for $\mathcal{E}(G)$, given in  (\ref{mcc}), $2\sqrt{m},$ is improved by the lower bound in (\ref{mcc3}) presented at Theorem \ref{CR}.
\begin{enumerate}
\item In \cite[Proof of Theorem 2]{Hofmeister}, for $T\neq S_{n}$ a tree with $n\geq 4$ vertices, its characteristic polynomial was presented:
\begin{equation}
    p_{T}(x)=x^{n-4}(x^{4}-(n-1)x^{2}+(n-3)).
\end{equation}
It is clear that $\kappa=n-4$, $\Upsilon_{4}(T)=n-3,$ and $m=n-1.$ Then, imposing the inequality
\begin{align}
\label{imhne}
    &2\sqrt{m}\leq \sqrt{2m+(n-\kappa)(n-\kappa-1)\left|\Upsilon_{n-\kappa}(T)\right|^{\frac{2}{n-\kappa}}},
\end{align}
    the inequality $(n-1)^2\leq 36(n-3)$ is obtained.

Therefore, for $4 \leq n \leq34$ the lower bound in (\ref{mcc3}) from Theorem \ref{CR}, improves the known lower bound $  2\sqrt{m}.$

\item Consider the join of two complete bipartite graphs, denoted as $G= K_{r_1,r_1} \vee K_{r_2,r_2}.$ Its spectrum (see e.g. \cite{CFMR2013}) is:
           \begin{align*}
                \sigma (G)
               =\{0^{2r_1-2+2r_2-2},-r_1,-r_2\} \cup \sigma(F)
           \end{align*}
where
\begin{equation*}
 F=
\begin{pmatrix}
 r_1 & 2\sqrt{r_1r_2}\\
 2\sqrt{r_1r_2}   & r_2
\end{pmatrix}   .
\end{equation*}
Then,
\begin{align*}
   \det(F)&=-3r_1r_{2}.
\end{align*}
Moreover,
\begin{align*}
    n&=2r_1+2r_2, & m&=r_1^2+r_2^2+4r_{1}r_{2},\\
    \kappa&= 2r_1+2r_2-4=n-4,  \ \text{and} &
    \Upsilon_{4}(G) &= (\beta_1)( \beta_2)(-r_1)(-r_2)=-3r_1^2r_2^{2},
\end{align*}
where $\beta_1$, $\beta_2$, are the eigenvalues of $F.$

By imposing the inequality in (\ref{imhne}), for $G$,
the following inequality is obtained:
\begin{align}
\label{GI}
    r_1^2+r_2^2 &\leq r_1r_2(6\sqrt{3}-4).
\end{align}
As, $6\sqrt{3}-4 \approx 6.39,$ if $5r_{1}=r_2$ the lower bound in (\ref{mcc3}) improves the lower bound $2\sqrt{m}.$

On the other hand, if one of the parameters is fixed, say $r_1$, from the inequality (\ref{GI}), the lower bound in (\ref{mcc3}) improves the lower bound $2\sqrt{m}$ whenever

\begin{equation*}
    r_1\left(0.16\right)\leq r_2 \leq r_1\left(6.23\right).
\end{equation*}
\item Let $G$ be a graph with $n$ vertices
and consider the generalized composition of the family of graphs
$\mathcal{F}= \{ \mathcal{H}_{1}, \ldots, \mathcal{H}_{n}\}$ where
$\mathcal{H}_{1}= \cdots = \mathcal{H}_{n}= \overline{K_t},$ $H= G[\mathcal{H}_{1},\dots, \mathcal{H}_{n}]$. Recall that, each vertex of $V(G)$ is assigned to the graph $\mathcal{H}_{j} \in \mathcal{F}$ (see \cite{CFMR2013, Schwenk74}). Then, from \cite[Theorem 5]{CFMR2013}

\begin{align*}
    \sigma(H)&=\mathop{\cup}\limits_{i=1}^{n} (\sigma(\mathcal{H}_{j}) \setminus \{0\}) \cup \sigma(tA(G))\\
    &=\{ 0^{n(t-1)} \} \cup \{ t\lambda : \lambda \in \sigma(A(G))\}.
\end{align*}


Therefore,
\begin{equation*}
    \mathcal{E}(H) =\sum_{\lambda \in \sigma(G)} |t\lambda|= t\mathcal{E}(G).
\end{equation*}

\noindent Thus, if $0 \in \sigma(G)$ has multiplicity $\kappa$ then $0 \in \sigma(H)$ has multiplicity  $ \kappa+n(t-1)$.
The following equalities are easy to compute:
\begin{align*}
    \overline{n}&=n(H)=nt, &
   \overline{m}&= m(H)=mt^{2},\\
   \overline{\kappa}&= \kappa(H)= n(t-1)+\kappa, &
    \Upsilon_{\overline{n}-\overline{\kappa}}(H) &= t^{n-\kappa} \Upsilon_{n-\kappa}(G).
\end{align*}
Suppose that $G$ is an $(n,m)$-graph with nullity  $\kappa$ such that inequality in (\ref{imhne}) holds.
Then, from previous equalities, $H= G[\mathcal{H}_{1},\dots, \mathcal{H}_{n}],$ is an  $(\overline{n},\overline{m})$- graph with nullity $\overline{\kappa}$ such that
   \begin{equation*}
      2\sqrt{\overline{m}}\leq  \sqrt{2\overline{m}+(\overline{n}-\overline{\kappa})(\overline{n}-\overline{\kappa}-1) \Upsilon_{\overline{n}-\overline{\kappa}}(H)^{\frac{2}{\overline{n}-\overline{\kappa}}}}
      ,
   \end{equation*} as
\begin{align*}
  &\sqrt{2\overline{m}+(\overline{n}-\overline{\kappa})(\overline{n}-\overline{\kappa}-1) \Upsilon_{\overline{n}-\overline{\kappa}}(H)^{\frac{2}{\overline{n}-\overline{\kappa}}}}\\
  &=\sqrt{2mt^2+(nt-(n(t-1)+\kappa))(nt-(n(t-1)+\kappa)-1)(t^{n-k}\Upsilon_{n-k}(G))^{\frac{2}{n-\kappa}}}\\
  &=\sqrt{2mt^2+(n-\kappa)(n-\kappa-1)t^{2}\Upsilon_{n-k}(G)^{\frac{2}{n-\kappa}}}\\ &=t\sqrt{2m+(t-\kappa))(n-\kappa)-1)\Upsilon_{n-k}(G)^{\frac{2}{n-\kappa}}}\\
  & \geq 2t\sqrt{m} = 2\sqrt{\overline{m}}.
\end{align*}
\end{enumerate}

\section{Proof of the main results.}

\noindent In this section we prove Theorems \ref{Robbiano-Infante-Carmona}, \ref{CR}, \ref{RCI2} and \ref{rCI} and the Corollaries \ref{RCI} and \ref{Carmona} described at Section 2.

\begin{proof}\textbf{of Theorem \ref{Robbiano-Infante-Carmona}}
Let $\alpha _{j_{1}}\geq \alpha _{j_{2}}\geq \cdots \geq \alpha
_{j_{n-\kappa }}$ be the non-zero eigenvalues of $R.$ It is
clear that%
\[
\mathcal{E}\left( R\right) =\left\vert \alpha _{j_{1}}\right\vert +\left\vert \alpha
_{j_{2}}\right\vert +\cdots +\left\vert \alpha _{j_{n-\kappa }}\right\vert .
\]%
Thus
\begin{eqnarray*}
\mathcal{E}\left( R\right) ^{2} &=&\left( \left\vert \alpha _{j_{1}}\right\vert
+\left\vert \alpha _{j_{2}}\right\vert +\cdots +\left\vert \alpha
_{j_{n-\kappa }}\right\vert \right) ^{2} \\
&=&\sum_{\ell =1}^{n-\kappa }\left\vert \alpha _{j_{\ell }}\right\vert
^{2}+\sum_{\ell _{1}\neq \ell _{2}}\left\vert \alpha _{j_{\ell
_{1}}}\right\vert \left\vert \alpha _{j_{\ell _{2}}}\right\vert .
\end{eqnarray*}%

\noindent Recall that $\Delta_{R}[{i_{1},i_{2},\ldots, i_{n-\kappa }]}$ denotes the $k \times k$ principal minor of $R.$ Since the geometric mean of a set
of positive numbers is not greater than the arithmetic mean, and the equality holds if and only if all of them are equal, we have:

\begin{eqnarray*}
\frac{1}{\left( n-\kappa \right) \left( n-\kappa -1\right) }\sum_{\ell
_{1}\neq \ell _{2}}\left\vert \alpha _{j_{\ell _{1}}}\right\vert \left\vert
\alpha _{j_{\ell _{2}}}\right\vert  &\geq& \left( \prod_{\ell _{1}\neq \ell
_{2}}\left\vert \alpha _{j_{\ell _{1}}}\right\vert \left\vert \alpha
_{j_{\ell _{2}}}\right\vert \right) ^{\frac{1}{\left( n-\kappa \right)
\left( n-\kappa -1\right) }} \\
&=&\left( \prod_{\ell =1}^{n-\kappa}\left\vert \alpha _{j_{\ell }}\right\vert \right) ^{%
\frac{2}{n-\kappa }} \\
&=&\left\vert \prod_{\ell=1 }^{n-\kappa}\alpha _{j_{\ell }}\right\vert ^{\frac{2}{%
n-\kappa }} \\
&=&\left\vert \Upsilon_{n-\kappa}\left(R\right) \right\vert ^{\frac{2}{n-\kappa }}.
\end{eqnarray*}

\noindent By the equality in (\ref{expansions}), the term $\left\vert \Upsilon_{n-\kappa}\left(R\right) \right\vert ^{\frac{2}{n-\kappa }} $ changes from a spectral invariant to a matrix invariant. Finally, the equality holds if and only if
 \begin{align}
 \label{alliquals}
     \left\vert \alpha _{j_{1}}\right\vert =\left\vert \alpha
_{j_{2}}\right\vert=\cdots =\left\vert \alpha _{j_{n-\kappa }}\right\vert.
 \end{align}
 From (\ref{alliquals}), attending to the definition of imprimitivity $h$ in \cite[Section III]{M}, we have $h=n- \kappa.$ Additionally, as $R$ is symmetric, its imprimitivity index must be $h=2.$ Therefore $\kappa=n-2.$ Moreover $R$ is cogredient (that is, permutationally similar), to a matrix of the form in (\ref{bm}) and as $\kappa=n-2,$ the block $S$ is a rank one matrix. By \cite[Theorem 4.2]{M} it is clear that, in this case, $\rho(R)$, (the spectral radius of $R$), and $-\rho(R)$ are the only nonzero eigenvalues of $R$.
 \end{proof}
\begin{remark}
Note that the equality \[
\left\vert \sum\limits_{1\leq i_{1}<i_{2}< \cdots < i_{n-\kappa } \leq n}
\Delta_{R}[i_{1},i_{2},\ldots, i_{n-\kappa}]\right\vert =
\left\vert \prod_{ l=1 }^{n-\kappa} \alpha_{j_{l}}\right\vert\]
is deduced considering the list of all (zero and nonzero) eigenvalues of $R$.
\end{remark}
\begin{proof}\textbf{of Theorem \ref{CR}}
\noindent The proof of the inequality is obtained following the same steps of the proof of Theorem \ref{Robbiano-Infante-Carmona} replacing the Hermitian matrix $R$ by the adjacency matrix of $G$. For the equality case, if $G$ is connected then $A(G)$ is irreducible and from the equality case in Theorem \ref{Robbiano-Infante-Carmona}, necessarily $G=K_{a,b}$. If $G$ is not connected then, by Theorem \ref{Robbiano-Infante-Carmona}, each connected component verifies the condition (\ref{alliquals}). Therefore, it is a complete bipartite graph and the described conditions for $G$ in the statement hold.
\end{proof}
\begin{proof}
\textbf{of Theorem \ref{RCI2}}
Let  $\alpha _{j_{1}}\geq \alpha _{j_{2}}\geq \cdots \geq \alpha
_{j_{n-\kappa }}$, with $\alpha _{j_{1}}=\rho$, be the non-zero eigenvalues of $R$. In \cite{{D-M-G}} it was proved that the real function $f(x)=x-1-\ln{x}, \quad x>0$ is a strictly increasing function for $x\geq 1$ and is decreasing in $0<x\leq 1$. Hence, $f(x)\geq f(1)=0,$ implying that $x\geq 1+\ln{x}, \quad x>0. $  Note that the equality holds if and only $x=1$. Using the above result, we get
\begin{eqnarray}
\label{Energia}
\begin{array}{lll}
    \mathcal{E}(R)&=& \displaystyle\rho + \sum _{ j=2 }^{ n-\kappa  }{ { \left|\alpha _{ { i }_{ j } }\right| } }\\
         & \geq & \displaystyle \rho + n-\kappa-1   +\sum _{ j=2 }^{ n-\kappa  }{ \ln{\left|\alpha _{ { i }_{ j } }\right| } }\\
         &=& \displaystyle \rho + n-\kappa-1   +\ln{\prod _{ j=2 }^{ n-\kappa  }{ \left|\alpha _{ { i }_{ j } }\right| }}\\
&=& \displaystyle \rho + n-\kappa-1   +\ln{{\left|\prod _{ j=2 }^{ n-\kappa  }{ \alpha _{ { i }_{ j } } }\right|}}\\
&=& \displaystyle \rho + n-\kappa-1   +\ln{\left|\Upsilon_{n-\kappa}(R)\right|}-\ln{\rho},
\end{array}
\end{eqnarray}
where the equality holds if and only if
\begin{align*}
     1=\left\vert \alpha _{j_{2}}\right\vert =\left\vert \alpha
_{j_{3}}\right\vert=\cdots =\left\vert \alpha _{j_{n-\kappa }}\right\vert.
 \end{align*}
Now, suppose that $R$ has largest eigenvalue greater than $1$ and $\rm{tr}(R)=0.$ Recalling that $\left\vert R\right \vert^{2}$ its the sum ob the squares of the absolute modulus of the eigenvalues of $R$, then the first equalities \ref{ce}., \ref{efe}. and \ref{kappa}. are obtained by searching solutions $\kappa$, $c$ and $f$ as function of $n$, $\rho$ and  $\left \vert R\right \vert$ in the following system:

 \begin{align*}
    1+c+f+\kappa&=n\\
    \rho+f-c&=0\\
    \rho^{2}+f+c&=\left\vert R\right \vert^{2}.
 \end{align*}

\noindent Now we discuss the case when the inequality (\ref{DasMGutmanCarmonaRobbiano4}) is strict. For the sufficient conditions \ref{subm1}. and \ref{subm2}. the interlacing of eigenvalues is used considering the smallest eigenvalues of $R$ and $R_{1},$ respectively (see, for instance \cite[Corollary 2.2]{haemersinterlacing}). As the smallest eigenvalue of $R_1$ in \ref{subm1}. is $-\sqrt{a^{2}+b^{2}}$ and imposing that this eigenvalue is smaller than $-1$ (note that, in this case its modulus is greater than $1$ and therefore $R$ doesn't fulfill the equality condition)
then $\sqrt{a^{2} +b^{2}}>1$.
 For the condition in \ref{subm2}. the Rayleigh quotient is used and the fact that the smallest eigenvalue of a symmetric matrix is at most a Rayleigh quotient of the matrix (\cite{haemersinterlacing, M}). Now, by noticing that either in  \ref{subm1}. or in
 \ref{subm2}. we impose that $R$ has the smallest eigenvalue not equal to $-1$, (using the same argument as before) the result follows.
 \end{proof}
\medskip
\textbf{Proof of Theorem \ref{rCI}}
The proof follows straightforward from the arguments used in the proof of Theorem \ref{RCI2} replacing the non-negative symmetric matrix $R$ by the adjacency matrix of the graph $G$.
For the equality case, and when $\rho=1,$ by Theorem \ref{CR} (attending that all the eigenvalues are of equal modulus) any connected component of $G$ has nonzero eigenvalues $1$ and $-1$ implying that they are isolated edges and therefore $G$ is the union of isolated vertices and isolated edges, that is $G=\lfloor \frac{n-\kappa}{2}\rfloor K_2 \cup \kappa K_1.$
On the other hand, if $\rho>1$ then $G$ must have a connected component with at least three vertices and one see that any induced subgraph with three vertices of this component must be a cycle otherwise it would be a path and by \ref{subm1}. and from Theorem \ref{RCI2}, $A(G)$ would have a submatrix of the form $R_1$ as in \ref{subm1}. (and using interlacing) the smallest eigenvalue of $G$ would not be $-1$. Therefore, if there exists a connected component of $G$ with at least three vertices, it must be a complete graph, and
then $G=K_{n-\ell}\cup \kappa K_{1}\cup \lfloor\frac{\ell-\kappa}{2}\rfloor K_2$ with
$\kappa\leq \ell \leq n-3.$
\begin{proof}
\textbf{of Corollary \ref{RCI}}
Recall that from the Rayleigh quotient $\rho=\alpha_1\geq  \frac{\mathbf{x}^{T} R \mathbf{x}}{\mathbf{x}^{T}\mathbf{x}}$ with equality if and only if $(\rho, \mathbf{x})$ is an eigenpair of $R$ (see e.g. \cite{M}). Recalling that, as in the proof of Theorem \ref{RCI2}, the real function $f(x)=x-1-\ln{x}, \quad x>0$ is strictly increasing for $x\geq 1$ and decreasing in $0<x\leq 1$ (\cite{{D-M-G}})
then $f(x)\geq f(1)=0,$ which implies $x\geq 1+\ln{x}, \quad x>0. $  Moreover, the real function $g(x)=x+ n-\kappa + \ln{\left|\Upsilon_{n-\kappa}(R)\right|}$ is a strictly increasing function, then the function $h= g \circ f $ is strictly increasing for $x\geq 1$. From the condition $\frac{\mathbf{x}^{T}R\mathbf{x}}{\mathbf{x}^{T}\mathbf{x}}\geq 1$ we have $h(\rho) \geq h(\frac{\mathbf{x}^{T}R\mathbf{x}}{\mathbf{x}^{T}\mathbf{x}}).$ Therefore, as $E(R)\geq h(\rho)$ as proved in Theorem \ref{RCI2}, the inequality follows.
If equality holds then for all nonzero eigenvalue of $R$,
$\alpha$, and different from the largest one the equality $\left|\alpha\right|=1+\ln{\left|\alpha\right|}$ occurs only when $\left|\alpha\right|=1$ implying that $\alpha=\pm1$ and $\rho=\frac{\mathbf{x}^{T}R\mathbf{x}}{\mathbf{x}^{T}\mathbf{x}},$ as $h$ is strictly increasing.
\end{proof}

\begin{proof}
\textbf{of Corollary \ref{Carmona}} Let $G_1$ be an induced $(n_{1}, m_{1})$-subgraph of $G$ with $n_1 \geq 1$. The proof follows directly from the proof of Corollary \ref{RCI} changing the non-negative symmetric matrix $R$ by the adjacency matrix of the graph $G$.
\noindent At this point recall that, if $\mathbf{x}$ is as in the statement of Remark \ref{regular} then
$\frac{\mathbf{x}^{T}A(G)\mathbf{x}}{\mathbf{x}^{T}\mathbf{x}}=\frac{2m_1}{n_1}\leq \lambda_{1}= \rho,$ with equality if and only if $G_1$ is a regular graph (see \cite{C-D-S1}, for example). Moreover,
the real function $g(x)=x+ n-\kappa \ln{\left|\Upsilon_{n-\kappa}(G)\right|}$  is strictly increasing, then the function $h= g \circ f $ is strictly increasing for $x\geq 1$. From the condition $\frac{2m_1}{n_1}\geq 1$ we have $h(\lambda_1) \geq h(\frac{2m_1}{n_1}),$
Therefore, as $E(R)\geq h(\rho)$ as proved in Theorem \ref{RCI2}, the inequality in (\ref{DasMGutmanCarmonaRobbiano3}) follows. If equality in (\ref{DasMGutmanCarmonaRobbiano3}) holds then for all nonzero eigenvalue (and different from the largest eigenvalue) $\lambda$ of $G$ the equality $\left|\lambda\right|=1+\ln{\left|\lambda\right|}$  occurs only when $\left|\lambda\right|=1$ implying that $\lambda=\pm1$ and $\lambda_{1}=\frac{2m_1}{n_1}.$ Therefore, $G_1$ is a regular connected component and then the graphs in the statement proceed.
\end{proof}
\vspace{0.6cm}

\section{An increasing sequence of lower bounds for the graph energy}

\noindent In this section we obtain an increasing sequence of lower bounds for the energy of graphs.
In \cite{H-T-W}, the authors built an increasing sequence, $\{\gamma^{(k)}\}_{k\geq 0}$ of lower bounds for $\lambda_{1}.$
Where,
\begin{equation}\label{seq1}
  \begin{array}{cc}
        \gamma^{(0)}=\sqrt{ \displaystyle\frac{ \displaystyle \sum_{i \in V(G)} d^2(i)}{n}} \\
            \\
        \gamma^{(1)}=\sqrt{ \frac{\displaystyle \sum_{i \in V(G)} d_2^2(i)}{\displaystyle \sum_{i \in V(G)} d^2(i)}} \\
        \vdots \\
        \gamma^{(k)}=\sqrt{ \frac{\displaystyle \sum_{i \in V(G)} d_{k+1}^2(i)}{\displaystyle \sum_{i \in V(G)} d^2_{k}(i)}}.
  \end{array}
\end{equation}
\noindent Then the following results were obtained.
\begin{theorem}{\cite{H-T-W}}
Let $G$ be a connected graph with largest eigenvalue $\lambda_{1}$ and $k \geq 0$. Then
$$ \lambda_1 \geq \gamma^{(k)}, $$
with equality if and only if $A^{k+2}(G)\mathbf{e}=\lambda_1^{2}A^{k}(G)\mathbf{e}.$
\end{theorem}

\begin{theorem}{\cite{H-T-W}}\label{T1}
Let G be a connected graph, then $\{\gamma^{(k)}\}_{k\geq 0}$ is an increasing sequence and $$\displaystyle \lim_{k\rightarrow \infty}\gamma^{(k)}=\lambda_{1}.$$
\end{theorem}
\begin{theorem}
\label{Carmona-Robbiano3}
Let $G$ be an $(n,m)$-graph with largest eigenvalue $\rho$ and $\eta(G)=\kappa.$ Let $G_1$ be an induced $(n_1,m_1)$-subgraph, that is, any connected component with spectral radius $\rho$ such that  $\frac{2m_1}{n_1}\geq 1$.
Let
 $\{\gamma_1^{(k)}\}_{k=0}^{\infty}$ be the increasing sequence defined in (\ref{seq1}) for $G_1$. Then $\{h(\gamma_1^{(k)})\}_{k=0}^{\infty}$ is an increasing sequence converging to $h (\rho)$ and, for all $ k\geq 0,$
\begin{equation}\label{C,R3}
    \mathcal{E}(G) \geq \gamma_1^{(k)}+n-\kappa-1+
    \ln{\left|\Upsilon_{n-\kappa}(G)\right|} -\ln \gamma_1^{(k)}.
\end{equation}
In particular, if $G_1$ is $r_1$-regular then
\begin{align}
\label{DasMGutmanCarmonaRobbiano9}
\mathcal{E}(G) \geq r_1+n-\kappa-1+\ln{\left|\Upsilon_{n-\kappa}(G)\right|}-\ln{r_1}.
\end{align}
If $\rho=1$ then equality holds if and only if  $G=\lfloor \frac{n-\kappa}{2}\rfloor K_2 \cup \kappa K_1.$ On the contrary, if $\rho >1 $ then equality holds if and only if $G=K_{n-\ell}\cup \kappa K_{1}\cup \lfloor\frac{\ell-\kappa}{2}\rfloor K_2$ with
$\kappa\leq \ell \leq n-3.$
\end{theorem}

 \begin{proof}
 Observe that $\gamma_1^{(k)}\geq 1$, for all $ k\geq 0.$
This is an immediate consequence of Theorem \ref{T1} and that  $$\gamma_{1}^{(0)}=\sqrt{\frac{\sum_{i \in V(G_1)} d_1^2(i)}{n_{1}}} \geq \sqrt{\frac{2m_1}{n_1}} \geq 1.$$ Since  $\{\gamma_1^{(k)}\}_{k=0}^{\infty}$ is an increasing sequence and converges to $\rho$ then, the first statement follows from the continuity of $h$. If equality holds in (\ref{C,R3}), then for all nonzero eigenvalue $\lambda$ (that is not equal to the largest eigenvalue of $G$) the equality $\left|\lambda\right|=1+\ln{\left|\lambda\right|}$ occurs only when $\left|\lambda\right|=1$ implying that $\lambda=\pm1.$ Additionally, if the equality occurs, $h(\gamma_{1}^{(k)})=\mathcal{E}(G)\geq h(\rho) \geq  h(\gamma_1^{(k)}).$ Therefore, $\mathcal{E}(G)= h(\rho) = h(\gamma_1^{(k)}), $ and we are in the conditions of Theorem \ref{rCI}. Therefore $G$ is as in the statement.
The inequality in (\ref{DasMGutmanCarmonaRobbiano9}) follows from the fact if $G_1$ is a $r_1$-regular graph, then $\gamma_1^{(k)}=r_1$, for all $k\geq 0.$
 \end{proof}

\noindent Recalling the result in (\ref{2}) obtained in \cite{D-M-G},
the result given in \cite{J} is here re-obtained considering $\kappa=0.$
\begin{corollary}
Let $G$ be a connected nonsingular graph of order $n$. Define
the sequence $\{\gamma^{(k)}\}_{k=0}^{\infty}$ as in (\ref{seq1}). Then
\begin{equation}\label{e1}
    \mathcal{E}(G) \geq \gamma^{(k)}+n-1+\ln|\det A|-\ln \gamma^{(k)},
\end{equation}
with $k\geq 0$
with equality if and only if $G=K_{n-\ell}\cup \lfloor\frac{\ell}{2}\rfloor K_2$ with
$0\leq \ell \leq n-2.$
\end{corollary}\label{Tg2}

\end{document}